\documentclass{svmult}
\usepackage{mathptmx,helvet,footmisc}
\usepackage{amsmath,amssymb,bbm,tikz,bm,verbatim}
\usepackage{stmaryrd}
\newcommand{\mc}[1]{\mathcal{#1}}
\newcommand{\bs}[1]{\boldsymbol{#1}}

\newcommand\N{\mathbb N}
\newcommand\sF{\mathcal F}
\newcommand\sS{\mathcal S}
\newcommand\sG{\mathcal G}
\newcommand\sC{\mathcal C}
\newcommand\q{\quad}
\newcommand\qq{\qquad}

\newcommand\lam{\lambda}
\newcommand\La{\Lambda}
\newcommand\bmu{{\bm{\mu}}}

\newcommand\gcev[1]{{\overset{{}_\shortleftarrow}{#1}}}
\newcommand\gvec[1]{{\overset{{}_\shortrightarrow}{#1}}}

\long\def\red#1{{\color{red}{#1}}}

\RequirePackage[colorlinks,citecolor=blue,urlcolor=blue]{hyperref}
\usepackage{tikz}
\usetikzlibrary{arrows.meta}
\usepackage{todonotes}
\usetikzlibrary{shapes}
\usepackage{color}

\renewcommand{\P}{\mathbb{P}}
\renewcommand{\E}{\mathbb{E}}
\newcommand\FF{\mathcal{F}} 

\newcommand\sP{\mathcal{P}}

\newcommand{\sss}[1]{\scriptscriptstyle{#1}}
\newcommand{\blank}[1]{}
\renewcommand{\qedsymbol}{\ensuremath{\blacksquare}}

\newcommand{\Eq}[1]{\overset{\sss{\bs{#1}}}{\sim}}
\newcommand{\Eqt}[2]{\overset{\sss{\bs{#1}_{#2}}}{\sim}}

\newcounter{thmcounter}
\addtocounter{thmcounter}{-1}

\newcommand\ol{\overline}

\newcommand\oo{\infty}
\newcommand\supp{\text{\rm supp}}
\def\resp{respectively}

\newcounter{mycount1}\newcounter{mycount2}\newcounter{mycount3}\newcounter{mycount}

\newenvironment{letlist}{\begin{list}{\rm(\alph{mycount3})}%
   {\usecounter{mycount3}\labelwidth=1cm\itemsep 0pt}}{\end{list}}

\newcommand\rank{\text{\rm rank}}

\title{Non-coupling from the past\\ \red{(Revised 16 October 2025)}} 

\author{Geoffrey R.\ Grimmett and Mark Holmes}
\institute{Geoffrey Grimmett \at Statistical Laboratory, Centre for
Mathematical Sciences, Cambridge University, Wilberforce Road,
Cambridge CB3 0WB, UK, \email{grg@statslab.cam.ac.uk}\\ and \\
School of Mathematics \&\ Statistics, The University of Melbourne, 
Parkville, VIC 3010, Australia
\and Mark Holmes \at School of Mathematics \&\ Statistics, The University of Melbourne, 
Parkville, VIC 3010, Australia. \email{holmes.m@unimelb.edu.au}
}

\begin{document}
\maketitle
\abstract{
The method of \lq coupling from the past'
permits exact sampling
from the invariant distribution of a Markov chain on a finite state space.
The coupling is successful whenever the stochastic dynamics are such that
there is coalescence of all trajectories. The issue of 
the coalescence or non-coalescence of
trajectories of a finite state space Markov chain 
is investigated in this note. The notion of the \lq coalescence number' 
$k(\mu)$ of a Markovian coupling  $\mu$ is introduced, and results
are presented concerning the set $K(P)$ of coalescence numbers
of couplings corresponding to a given transition matrix $P$.
\emph{\red{Note: This is a revision of the original published version, in which part of
Theorem \ref{thm:7} has been removed. A correction may be found in \cite[Thm 5.3]{GH}.}}}


\section{Introduction}\label{sec:1}
The method of \lq coupling from the past' (CFTP)
was introduced by Propp and Wilson \cite{PW2,PW3,WP1} in order to sample from the invariant distribution of an 
irreducible
Markov chain on a finite state space. It has attracted great interest
amongst theoreticians and practitioners, and there is an extensive associated literature
(see, for example, \cite{Huber,wil}).

The general approach of CFTP is as follows. Let $X$ be an irreducible  
Markov chain on a finite state space $S$ with transition matrix $P=(p_{i,j}:i,j\in S)$, 
and let $\pi$ be the unique invariant distribution 
(see \cite[Chap.\ 6]{GS} for a general account of the theory
of Markov chains). 

Let $\FF_S$ be the set of functions from $S$ to $S$, and
let $\sP_S$ be the set of all irreducible 
 stochastic
matrices on the finite set $S$.  We write $\N$ for the set $\{1,2,\dots\}$
of natural numbers, and $\P$ for a 
generic probability measure.

\begin{definition}
A probability measure $\mu$ on $\FF_S$ is \emph{consistent} with $P\in\sP_S$,
in which case we say that the pair $(P,\mu)$ is \emph{consistent},  if
\begin{equation}\label{eq:4}
p_{i,j} = \mu\bigl(\{f\in\FF_S: f(i)=j\}\bigr), \qq i,j\in S.
\end{equation}
Let $\mc{L}(P)$ denote the set of probability measures $\mu$ on $\FF_S$ 
that are consistent with $P\in\sP_S$.  
\end{definition}

Let $P\in \sP_S$ and $\mu\in\mc L(P)$. 
The measure $\mu$ is called a \emph{grand coupling} of
$P$.
Let $F=(F_s:s\in\N)$ be a vector of independent samples from $\mu$,  let 
$\gcev F_t$ denote the composition $F_1 \circ F_{2} \circ \dots \circ F_t$,
and define the \emph{backward coalescence time} 
\begin{equation}\label{eq:cotime}
C=\inf\bigl\{t:    \gcev F_t(\cdot)\text{ is a constant function}\bigr\}.
\end{equation}
We say that \emph{backward coalescence occurs} 
if $C<\oo$.
On the event $\{C<\oo\}$, $\gcev F_C$ may be regarded
as a random state.

The definition of coupling may seem confusing on first encounter. 
The function $F_1$ describes transitions during one step of the chain from time $-1$ to time $0$, as illustrated
in Figure \ref{fig:1}. If $F_1$ is not a constant function, we move back one step  in time to $-2$,
and consider the composition $F_1\circ F_2$. 
This process is iterated, moving one step back in time 
at each stage, until the earliest (random) $C$ such that
the iterated function $\gcev F_C$ is constant.
This $C$ (if finite) is the time to backward coalescence. 

Propp and Wilson proved the following fundamental theorem.

\begin{figure}\label{fig:1}
\centerline{\includegraphics[width=0.7\textwidth]{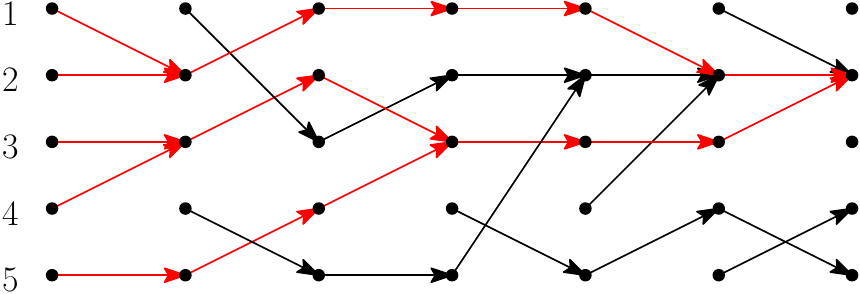}}
\caption{An illustration of coalescence of trajectories in CFTP with $|S|=5$.}
\end{figure}

\begin{theorem}[\cite{PW2}]\label{thm:pw}
Let  $P\in\sP_S$ and $\mu\in \mc{L}(P)$.
Either $\P(C<\oo)=0$ or $\P(C<\oo)=1$.
If it is the case that $\P(C<\oo)=1$, then the random state $\gcev F_C$ has law $\pi$.
\end{theorem}

Here are two areas of application of CFTP. In the first, one begins with a recipe 
for a certain probability measure $\pi$ on $S$, for example as the 
posterior distribution of a Bayesian analysis.  In seeking a sample from $\pi$,
one may find an aperiodic transition matrix $P$ having $\pi$ as unique invariant distribution, and
then run CFTP on the associated Markov chain. In a second situation
that may arise in a physical model, one 
begins with a Markovian dynamics with associated transition matrix $P\in\sP_S$, and
uses CFTP to sample from the invariant distribution. In the current work,
we shall assume that the transition matrix $P$ is specified, and that $P$ is
finite and irreducible.

Here is a summary of the work presented here.
In Section \ref{sec:2}, we discuss the phenomena of backward and forward 
coalescence, and we define the coalescence number of a Markov coupling. 
Informally, the coalescence number is the (deterministic) limiting number of un-coalesced trajectories of the coupling. 
Theorem \ref{thm:4} explains the relationship between the coalescence number
and the ranks of products of extremal elements in a convex representation of the stochastic matrix
$P$. The question is posed of determining the set $K(P)$ of coalescence numbers of couplings
consistent with a given $P$. A sub-family of couplings, termed \lq block measures', is studied
in Section \ref{sec:3}.  
These are couplings for which there is a fixed set of blocks (partitioning the state space), such that blocks of states are mapped to blocks of states, and such that coalescence occurs within but not between blocks. It is shown in Theorem \ref{lem:3}, via Birkhoff's convex representation theorem
for doubly stochastic matrices, that $|S|\in K(P)$ if and only if 
$P$ is doubly stochastic. Some further results about $K(P)$ are presented
in Section \ref{sec:4}.

\section{Coalescence of trajectories}\label{sec:2}

CFTP relies upon almost-sure backward coalescence, which is to say that $\P(C<\oo)=1$,
where $C$ is given in \eqref{eq:cotime}. 
For given $P\in \sP_S$,
the occurrence (or not) 
of coalescence depends on the choice of $\mu \in \mc L(P)$; see for example, Example \ref{eq:1}.

We next introduce the notion of \lq forward coalescence', which is to be considered as \lq \emph{coalescence}' but
with the difference that time runs
forwards rather than backwards. 
As before, let $P\in\sP_S$, $\mu\in\mc L(P)$, and let $F=(F_s: s \in \N)$ be  an independent sample from $\mu$.
For $i \in S$, define the Markov chain $X^i=(X_t^i: t \ge 0)$ by $X_t^i=\gvec F_t(i)$ where
$\gvec F_t=F_t \circ F_{t-1}\circ \dots \circ F_1$. Then 
	 $(X^i:i\in S)$ is
a family of coupled Markov chains, running forwards in time, each having transition matrix $P$,
and such that $X^i$ starts in state $i$.

The superscript $\shortrightarrow$ (\resp, $\shortleftarrow$)
is used to indicate that time is running forwards (\resp, backwards).
For $i,j\in S$, we say that $i$ and $j$ \emph{coalesce} if there exists $t$ such that  
$X_t^i = X_t^j$. We say that \emph{forward coalescence occurs} if,
for all pairs $i,j \in S$, $i$ and $j$ coalesce. The \emph{forward coalescence time}
is given by
\begin{equation}\label{eq:cotime2}
T=\inf\{t\ge 0: X_t^i=X_t^j \text{ for all } i,j \in S\}.
\end{equation}
Clearly, if $P$ is periodic then $T=\infty$ a.s.\ for any $\mu \in \mc{L}(P)$.  
A simple but important observation is that $C$ and $T$ have the same distribution.

\begin{theorem}\label{thm:basic}
Let $P\in \sP_S$ and $\mu\in\mc L(P)$.
The backward coalescence time $C$ and the forward coalescence time $T$ have the same distribution.
\end{theorem}

\begin{proof}
Let $(F_i: i \in \N)$ be an independent sample from $\mu$.
For $t \ge 0$, we have 
$$
\P(C \le t) = \P\bigl(\gcev F_t(\cdot) \text{ is a constant function}\bigr).
$$
By reversing the order of the functions $F_1,F_2,\dots,F_t$, 
we see that this equals $\P(T\le t)=\P(\gvec F_t(\cdot) \text{ is a constant function})$.
\qed\end{proof}

\begin{example}\label{eq:1}
Let $S=\{1,2,\dots,n\}$ where $n \ge 2$, and let $P_n=(p_{i,j})$
be the constant matrix with entries $p_{i,j}=1/n$ for  $i,j \in S$. 
Let $F=(F_i: i \in \N)$ be an independent sample from $\mu\in\mc L(P_n)$.
\begin{letlist}
\item If each $F_i$ is a uniform random permutation of $S$,
then $T\equiv \oo$ and $\gvec F_t(i)\ne \gvec F_t(j)$ for all $i \ne j$ and $t \ge 1$.
\item If  $(F_1(i):i \in S)$ are independent and uniformly distributed on $S$,
then $\P(T<\oo)=1$.
\end{letlist}
In this example, there exist measures $\mu \in \mc L(P_n)$
such that either {\rm(a)} a.s.\ no pairs of states coalesce, or {\rm(b)} a.s.\ forward
coalescence occurs.
\end{example} 

For $g \in \FF_S$, we let $\Eqt{g}{}$ be the equivalence relation on $S$ given by 
$i\Eqt{g}{}j$ if and only if $g(i)=g(j)$. For $f=(f_t: t \in \N)\subseteq \FF_S$
and $t \ge 1$, we write
$$
\gcev f_t=f_1\circ f_{2}\circ\dots\circ f_t,\q \gvec f_t=f_t\circ f_{t-1}\circ\dots\circ f_1.
$$
Let $k_t(\gcev{ f})$ (\resp, $k_t(\gvec{ f})$) denote the number of equivalence classes of the
relation $\Eqt{\gcev f}{t}$ (\resp, $\Eqt{\gvec f}{t}$).  
Similarly, we define the equivalence relation $\Eq{\gcev f}$ 
on $S$ by $i \Eq{\gcev f} j$ if and only if $i\Eqt{\gcev f}{t}j$ for some $t \in \N$, 
and we let $k(\gcev f)$ be the number of equivalence classes of $\Eq{\gcev f}$ (and similarly for $\gvec f$).  
We call $k(\gcev f)$  the \emph{backward} \emph{coalescence number} of $\gcev f$,
and likewise $k(\gvec f)$ the \emph{forward coalescence number} of $\gvec f$. 
The following lemma is elementary.

\begin{lemma}\label{lem0}
\mbox{}
\begin{letlist}
\item We have that $k_t(\gcev f)$  and $k_t(\gvec f)$ are monotone  non-increasing in $t$.
Furthermore, $k_t(\gcev f)=k(\gcev f)$ and $k_t(\gvec f)=k(\gvec f)$ 
for all large $t$.
\item Let $F=(F_s:s \in \N)$ be independent and identically distributed 
elements in $\sF_S$.
Then $k_t(\gcev F)$ and $k_t(\gvec F)$ are equidistributed, and similarly $k(\gcev F)$ and $k(\gvec F)$ are equidistributed.
\end{letlist}
\end{lemma}

\begin{proof}
(a) The first statement holds by consideration of the definition, and the second since  
$k(\gcev F)$ and $k(\gvec F)$ are integer-valued.

(b) This holds as in the proof of Theorem \ref{thm:basic}.
\qed\end{proof}

\section{Coalescence numbers}\label{sec:conos}

In light of Theorem \ref{thm:basic} and Lemma \ref{lem0},
we henceforth consider only Markov chains running in \emph{increasing positive time}.
\emph{Henceforth, expressions involving the word \lq coalescence'  
shall refer to \emph{forward} coalescence.}
Let $\mu$ be a probability measure on $\FF_S$, 
and let $\supp(\mu)$ denote the support of $\mu$.
Let 
$F=(F_s:s \in \N)$ be a vector of independent and identically distributed 
random functions, each with law $\mu$.  The law of $F$ is the product measure
$\bmu=\prod_{i\in \N}\mu$. The coalescence time $T$ is given by \eqref{eq:cotime2}, and 
the term \emph{coalescence number} refers to the quantities $k_t(\gvec F)$ and $k(\gvec F)$,
which we denote henceforth by $k_t(F)$ and $k(F)$, \resp.

\begin{lemma}
\label{lem:k_det}
Let $\mu$, $\mu_1$, $\mu_2$ be probability measures on $\FF_S$.
\begin{letlist}
\item Let $F=(F_s: s \in\N)$ be a sequence of independent and identically distributed 
functions each with law $\mu$.
We have that $k(F)$ is $\bmu$-a.s.\ constant,
and we write $k(\mu)$ for the almost surely  constant value of $k(F)$.
\item If $\supp(\mu_1) \subseteq \supp(\mu_2)$, then $k(\mu_1) \ge k(\mu_2)$.
\item If $\supp(\mu_1)=\supp(\mu_2)$, then $k(\mu_1)=k(\mu_2)$.
\end{letlist}
\end{lemma}

We call $k(\mu)$ the \emph{coalescence number} of $\mu$.

\begin{proof}
(a) For $j \in \{1,2,\dots,n\}$, let $q_j=\bmu(k(F)=j)$, and 
$k^*=\min\{j:q_j>0\}$.  Then 
\begin{equation}\label{eq:3}
\bmu(k(F)\ge k^*)=1.
\end{equation}  
Moreover, we choose $t \in \N$ such that 
$$
\kappa:=\bmu(k_t(F)=k^*) \q\text{satisfies} \q \kappa>0.
$$

For $m \in\N$, write $F^m=(F_{mt+s}: s\ge 1)$.
The event $E_{t,m}=\{k_t(F^m)=k^*\}$ depends only on $F_{mt+1},F_{mt+2}\dots,F_{(m+1)t}$.
It follows that the events $\{E_{t,m} : m\in \N\}$ are independent, and each occurs with probability $\kappa$.  
Therefore, almost surely at least one of these events occurs, and hence
$\bmu(k(F)\le k^*)=1$.   By \eqref{eq:3},
this proves the first claim.  

(b) Assume $\supp(\mu_1) \subseteq \supp(\mu_2)$, and let
$k^*_i$ be the bottom of the ${\bmu_i}$-support of $k( F)$. 
Since, for large $t$,  ${\bmu_1}(k_t(F)=k^*_1)>0$,
we have also that ${\bmu_2}(k_t(F)=k^*_1)>0$, whence $k^*_1 \ge k^*_2$.
Part (c) is immediate.
\qed\end{proof}

Whereas $k(F)$ is a.s.\ constant (as in Lemma \ref{lem:k_det}(a)),
the equivalence classes of $\Eq{\gvec F}$ need not themselves be a.s.\ constant.
Here is an example of this, preceded by some notation.

\begin{definition}\label{def:fn}
Let $f \in \sF_S$ where $S=\{i_1,i_2,\dots,i_n\}$ is a finite ordered set. We write $f=(j_1j_2\ldots j_n)$ if
$f(i_r)=j_r$ for $r=1,2,\dots,n$. 
\end{definition}

\begin{example}\label{ex:7}
Take $S=\{1,2,3,4\}$ and any 
	consistent pair $(P,\mu)$
	with $\supp(\mu)=\{f_1,f_2,f_3,f_4\}$, where
	\begin{equation*}
	f_1 = (3434),\q f_2 = (4334),\q 
	f_3 = (3412),\q f_4 = (3421).
	\end{equation*}
Then $k(\mu)=2$ but the equivalence classes 
of $\Eq{\gvec F}$ may be either $\{1,3\}$, $\{2,4\}$ or $\{1,4\}$, $\{2,3\}$,
each having a strictly positive probability.
	The four functions $f_i$ are illustrated diagrammatically in Figure \ref{random_coalesce}.
\end{example}

\def\x{4}
\def\y{8}
\begin{figure}
	\begin{center}
		\begin{tikzpicture}[scale=.8]
		\node (AA1) at (-.5,0) {1};
		\node (AA2) at (-.5,-1) {2};
		\node (AA3) at (-.5,-2) {3};
		\node (AA4) at (-.5,-3) {4};
		\node[scale=0.5,circle,fill=black] (A1) at (0,0) {};
		\node[scale=0.5,circle,fill=black] (A2) at (0,-1) {};
		\node[scale=0.5,circle,fill=black] (A3) at (0,-2) {};
		\node[scale=0.5,circle,fill=black] (A4) at (0,-3) {};
		\node[scale=0.5,circle,fill=black] (B1) at (2,0) {};
		\node[scale=0.5,circle,fill=black] (B2) at (2,-1) {};
		\node[scale=0.5,circle,fill=black] (B3) at (2,-2) {};
		\node[scale=0.5,circle,fill=black] (B4) at (2,-3) {};
		\draw[-Latex] (A1)--(B3);
		\draw[-Latex] (A2)--(B4);
		\draw[-Latex] (A3)--(B3);
		\draw[-Latex] (A4)--(B4);

		\node (CC1) at (-.5+\x,0) {1};
		\node (CC2) at (-.5+\x,-1) {2};
		\node (CC3) at (-.5+\x,-2) {3};
		\node (CC4) at (-.5+\x,-3) {4};
		\node[scale=0.5,circle,fill=black] (C1) at (0+\x,0) {};
		\node[scale=0.5,circle,fill=black] (C2) at (0+\x,-1) {};
		\node[scale=0.5,circle,fill=black] (C3) at (0+\x,-2) {};
		\node[scale=0.5,circle,fill=black] (C4) at (0+\x,-3) {};
		\node[scale=0.5,circle,fill=black] (D1) at (2+\x,0) {};
		\node[scale=0.5,circle,fill=black] (D2) at (2+\x,-1) {};
		\node[scale=0.5,circle,fill=black] (D3) at (2+\x,-2) {};
		\node[scale=0.5,circle,fill=black] (D4) at (2+\x,-3) {};
		\draw[-Latex] (C1)--(D4);
		\draw[-Latex] (C2)--(D3);
		\draw[-Latex] (C3)--(D3);
		\draw[-Latex] (C4)--(D4);
		
		\node (EE1) at (-.5+\y,0) {1};
		\node (EE2) at (-.5+\y,-1) {2};
		\node (EE3) at (-.5+\y,-2) {3};
		\node (EE4) at (-.5+\y,-3) {4};
		\node[scale=0.5,circle,fill=black] (E1) at (0+\y,0) {};
		\node[scale=0.5,circle,fill=black] (E2) at (0+\y,-1) {};
		\node[scale=0.5,circle,fill=black] (E3) at (0+\y,-2) {};
		\node[scale=0.5,circle,fill=black] (E4) at (0+\y,-3) {};
		\node[scale=0.5,circle,fill=black] (F1) at (2+\y,0) {};
		\node[scale=0.5,circle,fill=black] (F2) at (2+\y,-1) {};
		\node[scale=0.5,circle,fill=black] (F3) at (2+\y,-2) {};
		\node[scale=0.5,circle,fill=black] (F4) at (2+\y,-3) {};
		\draw[-Latex] (E1)--(F3);
		\draw[-Latex] (E2)--(F4);
		\draw[-Latex] (E3)--(F1);
		\draw[-Latex] (E4)--(F2);

		\node (GG1) at (-.5+3*\x,0) {1};
		\node (GG2) at (-.5+3*\x,-1) {2};
		\node (GG3) at (-.5+3*\x,-2) {3};
		\node (GG4) at (-.5+3*\x,-3) {4};
		\node[scale=0.5,circle,fill=black] (G1) at (0+3*\x,0) {};
		\node[scale=0.5,circle,fill=black] (G2) at (0+3*\x,-1) {};
		\node[scale=0.5,circle,fill=black] (G3) at (0+3*\x,-2) {};
		\node[scale=0.5,circle,fill=black] (G4) at (0+3*\x,-3) {};
		\node[scale=0.5,circle,fill=black] (H1) at (2+3*\x,0) {};
		\node[scale=0.5,circle,fill=black] (H2) at (2+3*\x,-1) {};
		\node[scale=0.5,circle,fill=black] (H3) at (2+3*\x,-2) {};
		\node[scale=0.5,circle,fill=black] (H4) at (2+3*\x,-3) {};
		\draw[-Latex] (G1)--(H3);
		\draw[-Latex] (G2)--(H4);
		\draw[-Latex] (G3)--(H2);
		\draw[-Latex] (G4)--(H1);
		\end{tikzpicture}
	\end{center}
	\caption{Diagrammatic representations of the four
	functions $f_i$ of Example \ref{ex:7}. The corresponding equivalence classes  
	are not $\bmu$-a.s.\ constant.}
	\label{random_coalesce}
\end{figure}
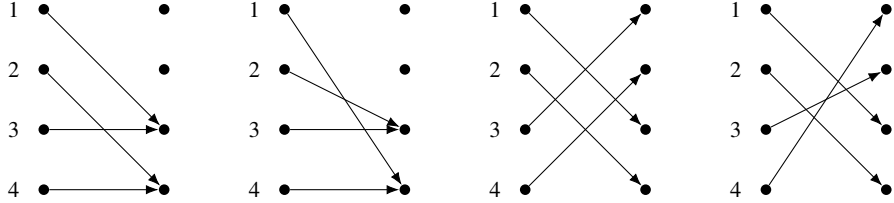

A probability measure $\mu$ on $\FF_S$ may be written in the form
\begin{equation}\label{eq:5}
\mu=\sum_{f\in\FF_S} \alpha_f \delta_f,
\end{equation}
where $\alpha$ is a probability mass function on $\FF_S$ with support $\supp(\mu)$,
and $\delta_f$ is the Dirac delta-mass on the point  $f \in \FF_S$.  
Thus, $\alpha_f>0$ if and only if $f \in\supp(\mu)$.
If $\mu\in\mc{L}(P)$, by \eqref{eq:4} and \eqref{eq:5},
\begin{equation}\label{eq:7}
P=\sum_{f\in \FF_S} \alpha_f M_f,
\end{equation}
where $M_f$ denotes the matrix 
\begin{equation}\label{eq:10}
M_f=(1_{\{f(i)=j\}}: i,j\in S),
\end{equation}
and $1_A$ is the indicator function of $A$.

Let $\Pi_S$ be the set of permutations of $S$. We denote also by $\Pi_S$
the set of matrices $M_f$ as $f$ ranges over the permutations of $S$.

\begin{theorem}\label{thm:4}
Let $\mu$ have the representation \eqref{eq:5}, and $|S|=n$. 
\begin{letlist}

\item We have that
\begin{equation}\label{eq:13}
k(\mu)=\inf\bigl\{\rank(M_{f_t}M_{f_{t-1}}\cdots M_{f_1}): 
f_1,f_2,\dots,f_t \in \supp(\mu),\ t \ge 1\bigr\}.
\end{equation}

\item There exists $T=T(n)$ such that the infimum in \eqref{eq:13}
is achieved for some $t$ satisfying $t\le T$.
\end{letlist}
\end{theorem}

\begin{proof}
(a) 
Let $F=(F_s: s \in \N)$ be drawn independently from $\mu$. 
Then
$$
R_t:= M_{F_t} M_{F_{t-1}}\cdots M_{F_1}
$$
is the matrix with $(i,j)$th entry $1_{\{\gvec F_t(i)=j\}}$.
Therefore, $k_t(F)$ equals the number of non-zero columns of $R_t$.
Since each row of $R_t$ contains a unique $1$, we have that $k_t(F) = \rank(R_t)$.
Therefore, $k(\mu)$ is the decreasing limit
\begin{equation}\label{eq:15}
k(\mu) = \lim_{t\to\oo} \rank(R_t)\qq\text{a.s.}
\end{equation}
Equation \eqref{eq:13} follows since $k(\mu)$ is integer-valued and deterministic.

(b) Since the rank of a matrix is integer-valued, the infimum in \eqref{eq:13} is attained.
The claim follows since, for given $|S|=n$, there are boundedly
many possible matrices $M_f$.
\qed\end{proof}

Let 
\begin{equation*}
K(P)=\bigl\{k: \text{ there exists $\mu \in \mc{L}(P)$ with }k(\mu)=k\bigr\}.
\end{equation*}
It is a basic question to ask: what can be said about $K$ as a function of $P$?
We first state a well-known result, based on ideas already in work of Doeblin \cite{Doeblin}.

\begin{lemma}\label{lem:8}
We have that $1\in K(P)$ if and only if $P\in\sP_S$ is aperiodic. 
\end{lemma}
\begin{proof}
	For $f\in\FF_S$, let $\mu(\{f\})=\prod_{i \in S}p_{i,f(i)}$.  
	This gives rise to $|S|$ chains with transition matrix $P$, starting from $1,2,\dots, n$,
	respectively, that evolve independently until they meet.  
	If $P$ is aperiodic (and irreducible) then all $n$ chains meet a.s.\ in finite time.  	

Conversely, if $P$ is periodic and $p_{i,j}>0$ then $i\ne j$, and $i$ and $j$ can never coalesce, 
implying $1\notin K(P)$. 
\qed\end{proof}

\begin{remark}\label{ex:5}
In a variety of cases of interest including, for example, the Ising 
and random-cluster models (see \cite[Exer.\ 7.3, Sect.\ 8.2]{pgs}), 
the set $S$ has a partial order, denoted $\le$.
For $P\in\sP_S$ satisfying the so-called FKG lattice condition, 
it is natural to seek $\mu\in\mc L(P)$
whose transitions preserve this partial order, and such $\mu$ may be constructed
via the relevant Gibbs sampler (see, for example, \cite[Sect.\ 6.14]{GS}). 
By the irreducibility of $P$, the trajectory starting at the least state of $S$ passes
a.s.\ through the greatest state of $S$. This implies that coalescence occurs, so that $k(\mu)=1$.
\end{remark}

\section{Block measures}\label{sec:3}

We introduce next the concept of a block measure, which is a strong form of the \emph{lumpability} of \cite{KS} and \cite[Exer.\ 6.1.13]{GS}.
	
	\begin{definition}
Let $P\in\sP_S$ and $\mu\in\mc L(P)$. For a partition 
$\sS = \{S_r: r=1,2,\dots,l\}$ of $S$ with $l=l(\sS)\ge 1$, we call $\mu$ 
an \emph{$\sS$-block measure} (or just a \emph{block measure}
with $l$ blocks) if
\begin{letlist}
\item
	for $f \in \supp(\mu)$, 
	there exists a unique permutation $\pi=\pi_f$ of $I:=\{1,2,\dots,l\}$
	such that, for $r\in I$,  $fS_r \subseteq S_{\pi(r)}$, and
\item $k(\mu)=l$.
\end{letlist}
\end{definition}
		
The action of an $\sS$-block measure $\mu$ is as follows. 
Since blocks are mapped a.s.\ to blocks, the measure $\mu$ of \eqref{eq:5}
induces a random permutation $\Pi$ of the blocks
which may be written as
\begin{equation}\label{new:41}
\Pi=\sum_{f\in \supp(\mu)} \alpha_f \delta_{\pi_f}.
\end{equation}
The condition $k(\mu)=l$ implies that 
\begin{equation}\label{new100}
\text{for $r\in I$ and $i,j\in S_r$, the pair $i$, $j$ coalesce a.s.,}
\end{equation}
so that the equivalence classes of $\Eq{\gvec{F}}$ are a.s.\ the blocks
$S_1,S_2,\dots,S_l$.
If, as the chain evolves, we observe only the evolution of the blocks,
we see a  Markov chain on $I$ with transition probabilities $\lambda_{r,s}=\P(\Pi(r)=s)$
which, since $P$ is irreducible, is itself irreducible.

Example \ref{ex:7} illustrates the existence of measures $\mu$
that are not block measures, when $|S|=4$.
On the other hand, we have the following lemma when $|S|=3$. 
For $P\in\sP_S$ and $\mu\in \mc L(P)$, let $\sC=\sC(\mu)$ be the set of possible coalescing pairs,
\begin{equation}\label{eq:pairs} 
\sC =\bigl\{\{i,j\}\subseteq S: i\ne j, \, \bmu(i,j \text{ coalesce})>0\bigr\}.
\end{equation}

\begin{lemma}
	\label{3blocks}
Let $|S|=3$ and $P\in\sP_S$.
If $(P,\mu)$ is consistent then $\mu$ is a block measure.
\end{lemma}

\begin{proof}
Let $S$, $(P,\mu)$ be as given.
If $\sC $ is empty then $k(\mu)=3$ and $\mu$ is a block measure with 3 blocks.

If $|\sC| \ge 2$, we have by the forthcoming Proposition 
\ref{prp:calP}(a,\,b) that
$k(\mu)\le 1$, so that $\mu$ is a block measure with $1$ block.

Finally, if $\sC$ contains exactly one element then we may assume, without loss of generality, 
that element is $\{1,2\}$.  By Proposition
\ref{prp:calP}(b),
we have $k(\mu)=1$, whence 
a.s.\ some pair coalesces.  By assumption only $\{1,2\}$ can coalesce, 
so in fact a.s.\ we have that $1$ and $2$  coalesce,
 and they do not coalesce with $3$.  
 Therefore, $\mu$ is a block measure with the two blocks $\{1,2\}$ and $\{3\}$.
\qed\end{proof}

We show next that, for $1\le k \le |S|$, there exists a consistent pair $(P,\mu)$ such that
$\mu$ is a block measure with $k(\mu)=k$.

\begin{lemma}\label{lem:5}
For $|S|=n\ge 2$ and $1\le k \le n$, there exists an aperiodic $P\in\sP_S$ such that $k \in K(P)$.
\end{lemma}

\begin{proof}
Let $\sS=\{S_r: r=1,2,\dots,l\}$ be a partition of $S$, and let $\sG\subseteq\FF_S$
be the set of all functions $g$ satisfying: 
there exists a permutation 
$\pi$ of $\{1,2,\dots,l\}$ such that,
for $r=1,2,\dots,l$, we have  $gS_r \subseteq S_{\pi(r)}$. 
Any probability measure $\mu$ on $\FF_S$
with support $\sG$ is an $\sS$-block measure.

Let $\mu$ be such a measure and let $P$ be the associated stochastic matrix on $S$,
given in \eqref{eq:4}.
For $i,j\in S$, there exists $g \in \sG$ such that $g(i)=j$. Therefore, $P$ is irreducible
and aperiodic.
\qed\end{proof}

We identify next the consistent pairs $(P,\mu)$ for which either $k(\mu)=|S|$ or $|S|\in K(P)$.

\begin{theorem}\label{lem:3}
Let $|S|=n \ge 2$ and $P\in \sP_S$. We have that
\begin{letlist}
\item 
$k(\mu)=n$ if and only if $\supp(\mu)$ contains only permutations of $S$,
\item $n \in K(P)$ if and only if $P$ is doubly stochastic.
\end{letlist}
\end{theorem}

Before proving this, we remind the reader of Birkhoff's theorem \cite{Birkh}
(sometimes attributed also
to von Neumann \cite{vN}).

\begin{theorem}[\cite{Birkh,vN}]\label{birk}
A stochastic matrix $P$ on the finite state space $S$ is doubly stochastic
if and only if it lies in the convex hull of the set $\Pi_S$ of permutation matrices.
\end{theorem}

\begin{remark}\label{rem:3}
We note that the simulation problem confronted by CFTP is trivial when 
$P$ is irreducible and doubly stochastic, since such $P$ are characterized as those
transition matrices with the uniform invariant distribution 
$\pi=(\pi_i=n^{-1}: i \in S)$. 
\end{remark}

\begin{proof}[Proof of Theorem \ref{lem:3}]
(a)
If $\supp(\mu)$ contains only permutations, then a.s.\  $k_t(F)=n$
for every $t\in\N$. Hence $n\in K(P)$.
If $\supp(\mu)$ contains a non-permutation,
then with positive probability $k_1(F)<n$ and hence $k(\mu)<n$.

(b)
By Theorem \ref{birk}, $P$ is doubly stochastic if and only if
it may be expressed as a  convex combination
\begin{equation}\label{eq:11}
P= \sum_{f\in \Pi_S}\alpha_f M_f,
\end{equation}
of permutation matrices $M_f$ (recall \eqref{eq:7} and \eqref{eq:10}).

If $P$ is doubly stochastic, let the $\alpha_f$ satisfy \eqref{eq:11}, and
let 
\begin{equation}\label{eq:14}
\mu=\sum_{f\in\Pi_S} \alpha_f\delta_f,
\end{equation}
as in \eqref{eq:5}.
Then 
$\mu\in\mc L(P)$, and $k(\mu)=n$ by part (a).

If $P$ is not doubly stochastic and $\mu\in\mc L(P)$, then $\mu$
has no representation of the form \eqref{eq:14}, so that $k(\mu)<n$ by
part (a). 
\qed\end{proof}

Finally in this section, we present a necessary
condition for $\mu$ to be an $\sS$-block measure; see Theorem \ref{thm:7} below.

Let $P\in\sP_S$, and let $\sS = \{S_r: r=1,2,\dots,l\}$ be a partition of $S$ with $l \ge 1$. 
For $r,s\in I:= \{1,2,\dots,l\}$ and $i \in S_r$, let
$$
\lambda_{r,s}^{(i)}= \sum_{j \in S_s} p_{i,j}.
$$
Since a block measure comprises  a transition operator on blocks, combined with a shuffling of states 
within blocks, it is necessary in order that $\mu$ be an $\sS$-block measure that
\begin{equation}\label{new:40}
\text{$\lambda_{r,s}^{(i)}$ is constant for $i\in S_r$.}
\end{equation}
When \eqref{new:40} holds, we write
\begin{equation}\label{new40a}
\lambda_{r,s} = \lambda_{r,s}^{(i)}, \qq i \in S_r.
\end{equation}
Under \eqref{new:40}, the matrix $\La=(\lambda_{r,s}: r,s\in I)$
is the irreducible transition matrix of the Markov chain derived from $P$
by observing the evolution of blocks, which is to say that
\begin{equation}\label{new40b}
\lambda_{r,s}= \mu(\Pi(r)=s), \qq r,s \in I,
\end{equation}
where $\Pi$ is given by \eqref{new:41}.
By Theorem \ref{lem:3}, we have that $l \in K(\La)$, and hence $\La$ is doubly stochastic, 
which is to say that
\begin{equation}\label{new:42}
\sum_{r\in I} \lambda_{r,s} = \sum_{r\in I}\sum_{j \in S_s} p_{i_r,j}=1, \qq s\in I,
\end{equation} 
where each $i_r$ is an arbitrarily chosen representative of the block $S_r$. By \eqref{new:40}, equation
\eqref{new:42} may be written in the form
\begin{equation}\label{new:43}
\sum_{i\in S}\sum_{j \in S_s} \frac1{|S_{r(i)}|}p_{i,j}=1, \qq r,s\in I,
\end{equation} 
where $r(i)$ is the index $r$ such that $i \in S_r$.  
We summarise the above discussion
in the following theorem.

\begin{theorem}\label{thm:7}
Let $S$ be a non-empty, finite set, let $P\in\sP_S$, and let $\sS = \{S_r: r=1,2,\dots,l\}$ be  a partition of $S$. 
If $\mu\in \mc L(P)$ is an $\sS$-block measure, then 
\eqref{new:40} and \eqref{new:43} hold, and also $k(\mu)=l$. 
\end{theorem}

\section{The set $K(P)$}\label{sec:4}

We begin with a triplet of conditions.

\begin{proposition}
	\label{prp:calP}
Let $S=\{1,2,\dots,n\}$ where $n \ge 3$, and
let $P\in\sP_S$ and $\mu\in\mc L(P)$. Let
$\sC=\sC(\mu)$ be the set of possible coalescing pairs, as in \eqref{eq:pairs}.
\begin{letlist}
\item $k(\mu)=n$ if and only if  $|\sC|=0$.
\item $k(\mu)=n-1$ if and only if $|\sC|=1$. 
\item
If $|\sC|$ comprises the single pair $\{1,2\}$, then $P$ satisfies
\begin{equation}\label{eq:new99}
\sum_{j=3}^n p_{1,j} = \sum_{j=3}^n p_{2,j} = \sum_{i=3}^n(p_{i,1}+p_{i,2}).
\end{equation}
\end{letlist}
\end{proposition}

\begin{proof}
(a)
 See Theorem \ref{lem:3}(a).

(b) By part (a), $k(\mu)\le n-1$ when $|\sC|=1$.  
It is trivial by definition of $k$ and $\mc{C}$ that, if $k(\mu)\le n-2$, 
then $|\mc{C}|\ge 2$. It suffices, therefore, to show that
$k(\mu)\le n-2$ when $|\sC|\ge 2$. Suppose that $|\sC|\ge 2$.
Without loss of generality we may assume that $\{1,2\}\in \sC$ and either that $\{1,3\}\in \sC$ or 
(in the case $n \ge 4$) that $\{3,4\}\in \sC$.
Let $F=(F_s:s\in \N)$ be an independent sample from $\mu$.
Let $M$ be the Markov time $M=\inf\{t>0: \gvec{F}_t(1)=\gvec{F}_t(2)=1\}$,
and write $J=\{M<\oo\}$.  By irreducibility, $\bmu(J)>0$,
implying that $k(\mu)\le n-1$.
Assume that
\begin{equation}\label{eq:new1002}
k(\mu)=n-1.
\end{equation}
We shall obtain a contradiction, and the conclusion of the lemma will follow.

\emph{Suppose first that  $\{1,2\}, \{1,3\}\in \sC$.}
Let $B$ be the event that there exists $i \ge 3$ such that $\gvec{F}_M(i)\in \{1,2,3\}$. 
On $B\cap J$, we have  $k(F)\le n-2$ a.s., since 
$$
\bmu(\text{at least $3$ states belong to coalescing pairs})>0.
$$ 
Thus $\mu(B\cap J)=0$ by  \eqref{eq:new1002}.
On $\ol{B}\cap J$, 
the $\gvec{F}_M(i)$, $i\ge 3$, are by \eqref{eq:new1002} a.s.\ distinct, 
and in addition take values in $S\setminus\{1,2,3\}$.  
Thus there exist $n-2$ distinct values of $\gvec{F}_M(i)$, $i\ge 3$, but at most $n-3$ values 
that they can take, which is impossible, whence $\mu(\ol{B}\cap J)=0$.  It follows that 
\begin{equation}\label{eq:new100}
0<\bmu(J)=\bmu(B\cap J) + \bmu(\ol B \cap J)=0,
\end{equation}
 a contradiction.

\emph{Suppose secondly that $\{1,2\}, \{3,4\}\in \sC$.}
Let $C$ be the event that either (i) there exists $i \ge 3$ such that
$\gvec{F}_M(i)\in\{1,2\}$, or (ii) $\{\gvec{F}_M(i) : i\ge 3\}\supseteq \{3,4\}$.  
On $C\cap J$, we have $k(F)\le n-2$ a.s.  
On $\ol C\cap J$, by \eqref{eq:new1002} the $\gvec{F}_M(i)$, $i\ge 3$, are a.s.\ distinct, 
and in addition take values in $S\setminus \{1, 2\}$ and no pair of them equals $\{3,4\}$.  
This provides a contradiction as in \eqref{eq:new100}.

(c)
Let $F_1$ have law $\mu$. Write  $A_i=\{F_1(i)\in \{1,2\}\}$,  and 
\begin{equation*}
M=|\{i\le 2: A_i \text{ occurs}\}|, \qq N=|\{i\ge 3: A_i \text{ occurs}\}|.
\end{equation*}
If $\mu(A_i\cap A_j)>0$ for some $i\ge 3$ and $j \ne i$, then 
$\{i,j\}\in \sC$, in contradiction of the assumption that $\sC$ comprises 
the singleton $\{1,2\}$.  
Therefore, $\mu(A_i\cap A_j)=0$ for all $i\ge 3$ and $j\ne i$, and hence
\begin{align}
\label{eq:new1001}\mu(N\ge 2)&=0, \\
\label{M1}\mu(M\ge 1,N=1)&=0.
\end{align}
By similar arguments,
\begin{align}
\label{M2}\mu(M<2,N=0)&=0,\\
\label{M3}\mu(M=1)&=0.
\end{align}
It follows that 
\begin{alignat*}{3}
\mu(N=1)&=\mu(N=1,M=0)  \qq&&\text{by \eqref{M1}}\\
&=\mu(M=0) &&\text{by \eqref{M2} and \eqref{eq:new1001}}\\
&=\mu(\ol A_1 \cap \ol A_2)\\
&=\mu(\ol A_r), \q r=1,2, &&\text{by \eqref{M3}.}
\end{alignat*}
Therefore,
\begin{equation*}
\mu(N=1)=\mu(\ol A_r)=\mu(F_1(r)\ge 3)=\sum_{j=3}^n p_{r,j}, \qq r=1,2.
\end{equation*}
By \eqref{eq:new1001},
\begin{equation*}
\mu(N= 1)=\mu(N) =\sum_{i=3}^n\mu(A_i)=\sum_{i=3}^n(p_{i,1}+p_{i,2}),
\end{equation*}
where $\mu(N)$ is the mean value of $N$. This yields \eqref{eq:new99}.
\qed\end{proof}

The set $K(P)$ can be fairly sporadic, as illustrated in the next two examples.

\begin{example}\label{ex:10}
	Consider the matrix 
	\begin{align}\renewcommand\arraystretch{1.2}
	P=\begin{pmatrix}
	\frac12 & \frac12 & 0\\
	0 & \frac12  & \frac12\\
	\frac12 & 0 & \frac12
	\end{pmatrix}.
	\end{align}
	Since $P$ is doubly stochastic, by Theorem \ref{lem:3}(a),  
	there exists $\mu\in\mc L(P)$ such that $k(\mu)=3$ (one
	may take $\mu(123)=\mu(231)=\frac12$).  By Lemma \ref{lem:8}, we have that $1 \in K(P)$,
	and thus $\{1,3\}\subseteq K(P)$.  We claim that $2\notin K(P)$, and we show this as follows.
	
Let $\mu \in \mc{L}(P)$, with $k(\mu)<3$, so that  $|\sC|\ge 1$. 
	There exists no permutation of $S$ for which the matrix $P$ satisfies \eqref{eq:new99}, 
	whence $|\sC|\ge 2$ by	Proposition \ref{prp:calP}(c).
	By parts (a,\,b) of that proposition, $k(\mu)\le 1$. In conclusion, $K(P)=\{1,3\}$.
	\end{example}

\begin{example}\label{ex:11}
	Consider the matrix
\begin{align}\renewcommand \arraystretch{1.2}
P=\begin{pmatrix}
\frac12 & \frac12 & 0 & 0\\
0 & \frac12  & \frac12 & 0\\
0 & 0 & \frac12 & \frac12\\
\frac12 & 0& 0 & \frac12
\end{pmatrix}.
\end{align}	
We have, as in Example \ref{ex:10}, that $\{1,4\}\subseteq K(P)$.  
Taking $$
\mu(1234)=\mu(2244)=\mu(1331)=\mu(2341)=\tfrac14
$$ 
reveals that $2\in K(P)$, and indeed $\mu$ is a block measure with blocks $\{1,2\}$, $\{3,4\}$.  
As in Example \ref{ex:10}, we have that $3\notin K(P)$, so that $K(P)=\{1,2,4\}$.
\end{example}

We investigate in greater depth the transition matrix on $S$ with equal entries.
Let $|S|=n \ge 2$ and let $P_n=(p_{i,j})$ satisfy 
$p_{i,j}=n^{-1}$ for $i,j\in S=\{1,2,\dots,n\}$.

\begin{theorem}\label{thm:6}
For $n \ge 2$ there exists a block measure $\mu\in\mc L(P_n)$ with
$k(\mu)=l$ if and only if $l \mid n$. In particular,  $K(P_n)\supseteq \{l: l\mid n\}$.
For $n \ge 3$, we have $n-1\notin K(P_n)$.
\end{theorem}

We do not know whether $K(P_n)=\{l:l\mid n\}$, and neither do we know if
there exists  $\mu\in\mc L(P_n)$ that is not a block measure.

\begin{proof}
Let $n \ge 2$. By Lemma \ref{lem:8} and Theorem \ref{lem:3}, we have that $1,n\in K(P_n)$.
It is easily seen as follows that $l \in K(P_n)$ whenever $l \mid n$. Suppose
$l\mid n$ and $l \ne 1,n$. 
Let 
\begin{align}
S_r&=(r-1)n/l+\{1,2,\dots,n/l\}\nonumber\\
&=\{(r-1)n/l +1,(r-1)n/l+2,\dots, rn/l\}, \qq r=1,2,\dots,l.
\end{align}
We describe next a measure $\mu\in\mc L(P_n)$. 
Let $\Pi$ be a uniformly chosen permutation
of $\{1,2,\dots,l\}$. For $i\in S$, let $Z_i$ be chosen uniformly at random
from $S_{\Pi(i)}$, where the $Z_i$ are conditionally independent given $\Pi$.
Let $\mu$ be the block measure governing the vector $Z=(Z_i:i\in S)$. By symmetry,
\begin{equation*}
q_{i,j} := \mu\bigl(\{f\in\FF_S: f(i)=j\}\bigr), \qq i,j\in S,
\end{equation*}
is constant for all pairs $i,j\in S$. Since $\mu$ is a probability measure, 
$Q=(q_{i,j})$ has row sums $1$, whence $q_{i,j}=n^{-1}=p_{i,j}$,
and therefore $\mu\in \mc L(P_n)$. By examination of $\mu$,
$\mu$ is an $\sS$-block measure.

Conversely, suppose there exists an $\sS$-block measure 
$\mu\in \mc L(P_n)$ with corresponding partition
$\sS=\{S_1,S_2,\dots,S_l\}$ with index set $I=\{1,2,\dots,l\}$. 
By Theorem \ref{thm:7}, equations  \eqref{new:40} and \eqref{new:43} hold.
By \eqref{new:40}, the matrix $\La=(\lambda_{r,s}: r,s \in I)$ satisfies
\begin{equation}\label{eq:31}
\lam_{r,s}= \frac {|S_s|}n, \qq r,s\in I.
\end{equation}
By 
\eqref{new:43},
$$
\frac{|S_s|}{|S_r|}=1,\qq s,r\in I,
$$
whence $|S_s|=n/l$ for all $s\in I$, and in particular $l \mid n$.

Let $n \ge 3$. We prove next that $k(\mu)\ne n-1$ for $\mu \in \mc L(P_n)$.
Let $\sC=\sC(\mu)$ be given as in \eqref{eq:pairs}.  
By Proposition \ref{prp:calP}(b), it suffices to prove that
$|\sC|\ne 1$.
Assume on the contrary that $|\sC|=1$, and suppose without loss of generality that $\sC$ contains the singleton 
pair $\{1,2\}$.  With $P=P_n$, the necessary condition \eqref{eq:new99} becomes
$$
(n-2)\frac 1n = (n-2)\frac 2n,
$$
which is false when $n \ge 3$. Therefore, $|\sC|\ne 1$, and the proof is complete.
\qed\end{proof}

\section*{Acknowledgements}
The authors thank Tim Garoni, Wilfrid Kendall, and an anonymous referee 
for their comments. 
MH was supported by Future Fellowship FT160100166 from the Australian Research Council.
\red{Tabet Aoul Alaa Eddine Islem kindly indicated an error in an earlier version of Theorem \ref{thm:7},
duly corrected in \cite[Thm 5.3]{GH}.}

\end{document}